\documentclass{amsart}

\usepackage{amsfonts}
\usepackage{amssymb}
\usepackage{enumerate}
\usepackage{amsmath}
\usepackage{amsthm}
\usepackage{graphicx}
\usepackage[english]{babel}

\numberwithin{equation}{section}

\newtheorem{theorem}{Theorem}[section]

\newtheorem{question}[theorem]{Question}

\newtheorem{fact}[theorem]{Fact}
\theoremstyle{definition}

\newtheorem{remark}[theorem]{Remark}
\newtheorem{definition}[theorem]{Definition}
\newtheorem{notation}[theorem]{Notation}

\newtheorem*{K}{Theorem \ref{t:cont}}
\newtheorem*{m1}{Theorem \ref{t:main1}}
\newtheorem*{m2}{Theorem \ref{t:main2}}

\newtheorem*{qq}{Question \ref{q:Elekes}}

\DeclareMathOperator{\diam}{diam}
\DeclareMathOperator{\dist}{dist}
\DeclareMathOperator{\Lip}{Lip}

\DeclareMathOperator{\graph}{graph}
\DeclareMathOperator{\inter}{int}
\DeclareMathOperator{\Fix}{Fix}

\title{Metric spaces admitting only trivial weak contractions}

\author{Rich\'ard Balka}
\address{Alfr\'ed R\'enyi Institute of Mathematics, Hungarian Academy of Sciences,
PO Box 127, 1364 Budapest, Hungary}
\email{balka.richard@renyi.mta.hu}
\thanks{Supported by the
Hungarian Scientific Research Fund grant no.~72655.}



\date{}

\begin{document}

\begin{abstract} If $(X,d)$ is a metric space then the map $f\colon X\to X$
is defined to be a weak contraction if $d(f(x),f(y))<d(x,y)$ for all $x,y\in X$,
$x\neq y$. We determine the simplest non-closed sets
$X\subseteq \mathbb{R}^n$ in the sense of descriptive set theoretic
complexity such that every weak contraction $f\colon X\to X$ is
constant. In order to do so, we prove that there exists a non-closed
$F_{\sigma}$ set $F\subseteq \mathbb{R}$ such that every weak contraction
$f\colon F\to F$ is constant. Similarly, there exists a non-closed
$G_{\delta}$ set $G\subseteq \mathbb{R}$ such that every weak contraction
$f\colon G\to G$ is constant. These answer questions of M. Elekes.

We use measure theoretic methods, first of all the concept of
generalized Hausdorff measure.
\end{abstract}

\maketitle

\section{Introduction}

We use the following descriptive set theoretical notation in this section.

\begin{notation} The class of open, closed, $F_{\sigma}$, and $G_{\delta}$ sets are denoted by
$\Sigma^{0}_1$, $\Pi^{0}_1$, $\Sigma^0 _2$, and $\Pi^0 _2$, respectively. The simultaneously
$F_{\sigma}$ and $G_{\delta}$ sets are denoted by $\Delta^0 _2$.
\end{notation}

M. Elekes \cite{E} introduced the next definition.
\begin{definition} We say that the metric space $X$ possesses the
\emph{Banach Fixed Point Property (BFPP)} if every contraction
$f\colon X\to X$ has a fixed point.
\end{definition}
The Banach Fixed Point Theorem implies that every complete metric
space has the BFPP. E. Behrends \cite{B} pointed out that the
converse implication does not hold. He presented the following
example, which he referred to as `folklore'.
\begin{theorem} Let $X=\graph\left(\sin(1/x)|_{(0,1]}\right)$.
Then $X\subseteq \mathbb{R}^2$ is a non-closed simultaneously $F_{\sigma}$
and $G_{\delta}$ set possessing the Banach Fixed Point Property.
\end{theorem}
M. Elekes \cite{E} described the simplest non-closed sets having the
BFPP in the sense of descriptive set theoretic complexity. He proved
the following theorems.

\begin{theorem}[M. Elekes] \label{t:Elekes1}
Every open subset of $\mathbb{R}^n$ with the Banach Fixed Point Property is
closed. Every simultaneously $F_{\sigma}$ and $G_{\delta}$ subset of
$\mathbb{R}$ with the Banach Fixed Point Property is closed.
\end{theorem}

\begin{theorem}[M. Elekes]  \label{t:Elekes2} There exist non-closed $F_{\sigma}$
and non-closed $G_{\delta}$ subsets of $\mathbb{R}$ with the Banach Fixed
Point Property.
\end{theorem}

The above three theorems answer the question about the lowest
possible Borel classes of $\mathbb{R}^n$ having a non-closed element with
the BFPP. In the language of descriptive set theory, if $n\geq 2$
then $\Delta^0_2$ is the best possible class, since there are no
$\Sigma^0_1$ and $\Pi^0_1$ examples. If $n=1$ then $\Sigma^0_2$ and
$\Pi^0_2$ are possible, but $\Delta_2^0$ is not.

\bigskip

Note that if every weak contraction $f\colon X\to X$ is constant
then $X$ has the BFPP. There are infinite complete metric spaces
that admit only trivial weak contractions, for example the metric
spaces $X=\mathbb{Z}\times \{0\}^{n-1}\subseteq \mathbb{R}^n$ clearly have this
property (there is a non-degenerate connected compact example in
$\mathbb{R}^n$ for every $n\geq 2$, see later). Therefore it is natural
to ask the following question.

\begin{question}[M. Elekes] \label{q:Elekes} What are the lowest possible Borel classes of $\mathbb{R}^n$ having
a non-closed element $X$ such that every weak contraction
$f\colon X\to X$ is constant?
\end{question}

The main goal of our paper is to answer Question \ref{q:Elekes}.

\bigskip

On the one hand, Theorem \ref{t:Elekes1} yields that there are no
$\Sigma^0_1$ and $\Pi^0_1$ examples in the cases $n\geq 2$.

On the other hand, T.~Dobrowolski \cite{D} pointed out the connection between our question
and the so called \emph{Cook continuum}; that is, a non-degenerate connected compact topological space $C$
such that every continuous map $f\colon C\to C$ is either constant or the identity. It was named after H.~Cook \cite{C},
who first constructed such an object. Cook's example cannot be embedded in $\mathbb{R}^2$, only in $\mathbb{R}^3$.
Later T.~Ma\'ckowiak \cite[Cor.~32.]{Ma1} has shown that there exists an arc-like (snake-like) Cook continuum, and arc-like continua are
embeddable in the plane by \cite[Thm.~4.]{Bi}. The next theorem is straightforward,
it follows that the answer is $\Delta^0_2$ if $n\geq 2$.

\begin{theorem}[Ma\'ckowiak, Dobrowolski] Let $X=C\setminus\{c_0\}$, where $C\subseteq \mathbb{R}^2$ is a Cook continuum and
$c_0\in C$ is arbitrary. Then $X\subseteq \mathbb{R}^2$ is
non-closed, simultaneously $F_{\sigma}$ and $G_{\delta}$, and every
weak contraction $f\colon X\to X$ is constant.
\end{theorem}

If $n=1$ then Theorem \ref{t:Elekes1} implies that there is no
$\Delta_2^0$ example for Question \ref{q:Elekes}. In the positive
direction M. Elekes obtained the following partial result.

\begin{theorem}[M. Elekes] \label{t:Elekes3} There exists a non-closed $G_{\delta}$ set
$G\subseteq \mathbb{R}$ such that every contraction $f\colon G\to G$ is
constant.
\end{theorem}

The proof of Theorem \ref{t:Elekes3} is based on the following
theorem, that is interesting in its own right.

\begin{theorem}[M. Elekes] \label{t:Etyp} For the generic compact set $K\subseteq \mathbb{R}$ (in the sense of Baire category)
for any contraction $f\colon K\to \mathbb{R}$ the set $f(K)$ does not contain a non-empty relatively open subset of $K$.
\end{theorem}

In order to answer Question \ref{q:Elekes} it is enough to show that
there are non-closed $\Sigma^0_2$ and $\Pi^0_2$ subsets of
$\mathbb{R}$ that admit only trivial weak contractions. Therefore we prove
the following theorems.

\begin{m1}[Main Theorem, $F_{\sigma}$ case] There exists a non-closed $F_{\sigma}$
set $F\subseteq \mathbb{R}$ such that every weak contraction $f\colon F\to
F$ is constant.
\end{m1}

\begin{m2}[Main Theorem, $G_{\delta}$ case] There exists a non-closed $G_{\delta}$
set $G\subseteq \mathbb{R}$ such that every weak contraction $f\colon G\to
G$ is constant.
\end{m2}

The heart of the proof is the following theorem, that is a partial,
measure theoretic analogue of Theorem \ref{t:Etyp}. For a gauge function $h$ let us denote by $\mathcal{H}^{h}$ the $h$-Hausdorff measure.
\begin{K}[simplified version] There exists a compact set $K\subseteq \mathbb{R}$ and a continuous gauge function $h$ such
that $0<\mathcal{H}^{h}(K)<\infty$, and for every weak contraction $f\colon
K\to \mathbb{R}$ we have $\mathcal{H}^{h} \left(K\cap f(K)\right)=0$.
\end{K}
Based on this paper, A. M\'ath\'e and the author show
in \cite{BM} the following more general theorem. If $X$ is a Polish space, then the generic compact set $K\subseteq X$ is
either finite or there is a continuous gauge function $h$ such that $0<\mathcal{H}^{h}(K)<\infty$, and for every weak contraction $f\colon
K\to X$ we have $\mathcal{H}^{h} \left(K\cap f(K)\right)=0$. If $X$ is perfect, then the generic compact set $K\subseteq X$ is infinite,
so the first case does not occur. This is the measure theoretic analogue of Theorem
\ref{t:Etyp}, which also answers a question of C.~Cabrelli, U.~B.~Darji, and U.~M.~Molter.
This is the reason why we will work in Polish spaces instead of $\mathbb{R}$.

The structure of the paper will be as follows. In the Preliminaries
section we introduce some notation and definitions. In Section
\ref{s:K} we define balanced compact sets in a Polish space $X$, and we prove its existence if $X$ is uncountable.
In Section \ref{s:gauge} we show that every balanced compact set $K\subseteq X$ has a continuous gauge function $h$ such that
$0<\mathcal{H}^{h}(K)<\infty$. In Section
\ref{s:proof} we show that $\mathcal{H}^{h} \left(K\cap f(K)\right)=0$ for every weak contraction $f\colon K\to X$, which completes the proof of Theorem~\ref{t:cont}.
In Section \ref{s:main} we prove our Main Theorems based on
Theorem~\ref{t:cont} and ideas from \cite{E}.

\section{Preliminaries}

Let $(X,d)$ be a metric space, and let $A,B\subseteq X$ be arbitrary
sets. We denote by $\inter A$ and $\diam A$ the interior and the diameter of $A$,
respectively. We use the convention $\diam \emptyset = 0$. The \emph{distance} of the sets $A$ and
$B$ is defined by $\dist (A,B)=\inf \{d(x,y): x\in A, \, y\in B\}$. The function $h\colon [0,\infty)\to [0,\infty)$ is defined to be a
\emph{gauge function} if it
is non-decreasing, right-continuous, and $h(x)=0$ iff $x=0$. For all
$A\subseteq X$ and $\delta>0$ consider
\begin{align*}
\mathcal{H}^{h}_{\delta}(A)&=\inf \left\{ \sum_{i=1}^\infty
h\left(\diam A_{i}\right): A \subseteq \bigcup_{i=1}^{\infty} A_{i},
\, \forall i
\diam A_i \leq \delta \right\}, \\
\mathcal{H}^{h}(A)&=\lim_{\delta\to 0+}\mathcal{H}^{h}_{\delta}(A).
\end{align*}
We call $\mathcal{H}^{h}$ the \emph{$h$-Hausdorff measure}. For more information on
these concept see \cite{Ro}.

A metric space $X$ is \emph{perfect} if it has no isolated points.
A metric space $X$ is \emph{Polish} if it is complete and separable.

Given two metric spaces $(X,d_{X})$ and $(Y,d_{Y})$, a function
$f\colon X\to Y$ is called \emph{Lipschitz} if there is a constant $C
\in \mathbb{R}$ such that $d_{Y}(f(x_{1}),f(x_{2}))\leq C \cdot
d_{X}(x_{1},x_{2})$ for all $x_{1},x_{2}\in X$. The smallest such
constant $C$ is the \emph{Lipschitz constant} of $f$ and denoted by
$\Lip(f)$. If $\Lip(f)\leq 1$ then $f$ is a \emph{$1$-Lipschitz map},
if $\Lip(f)<1$ then $f$ is a \emph{contraction}. We say that $f$ is a \emph{weak contraction} if
$d_{Y}(f(x_{1}),f(x_{2}))<d_{X}(x_{1},x_{2})$ for all $x_{1},x_{2}\in X$, $x_1\neq x_2$.

Stand $\lambda$ for the Lebesgue measure of $\mathbb{R}$. Let us denote the positive odd numbers by $2\mathbb{N}+1$.

\section{The definition and existence of balanced compact sets} \label{s:K}

\begin{definition} If $a_n$ $(n\in \mathbb{N}^+)$ are positive integers
then let us consider for all $n\in \mathbb{N}^+$,
$$\mathcal{I}_{n}=\prod_{k=1}^{n}\{1,\dots,a_k\} \quad \textrm{and} \quad \mathcal{I}=\bigcup_{n=1}^{\infty} \mathcal{I}_{n}.$$
We say that a map $\Phi \colon 2\mathbb{N}+1 \to \mathcal{I}$
is an \emph{index function  according to the sequence $\langle a_n \rangle$}
if it is surjective and $\Phi(n) \in \bigcup _{k=1}^{n} \mathcal{I}_{k}$ for every odd $n$.
\end{definition}

\begin{definition} \label{d:balanced} Let $X$ be a Polish space.
A compact set $K\subseteq X$ is \emph{balanced} if it is of the form
\begin{equation} \label{eq:defK} K=\bigcap _{n=1}^{\infty}\left(\bigcup_{i_1=1}^{a_1} \cdots  \bigcup_{i_n=1}^{a_n}C_{i_1 \dots i_n} \right),
\end{equation}
where $a_{n}$ are positive integers and $C_{i_1\dots i_n}\subseteq X$ are non-empty closed sets with
the following properties. There are positive reals $b_n$ and there is an
index function $\Phi \colon 2\mathbb{N} +1 \to \mathcal{I}$ according to the sequence $\langle a_n\rangle$
such that for all $n\in \mathbb{N}^+$ and $(i_1,\dots ,i_{n}),(j_1,\dots ,j_{n})\in \mathcal{I}_{n}$

\begin{enumerate}[(i)]
\item \label{00} $a_1\geq 2$ and $a_{n+1}\geq n a_{1}\cdots a_{n}$,
\item \label{01} $C_{i_{1}\dots i_{n+1}}\subseteq
C_{i_1 \dots i_{n}}$,
\item \label{02} $\diam C_{i_{1} \dots i_n}\leq b_n$,
\item  \label{03} $\dist(C_{i_1 \dots i_n},C_{j_1\dots j_n})>2b_n$
if $(i_1,\dots ,i_n)\neq (j_1, \dots ,j_n)$.
\item \label{04}
If $n$ is odd, $C_{i_1 \dots i_{n}}\subseteq C_{\Phi(n)}$ and $C_{j_1 \dots j_{n}}
\nsubseteq C_{\Phi(n)}$, then for all $s,t
\in  \{1, \dots ,a_{n+1}\}$, $s\neq t$ we have
$$\dist\left(C_{i_1\dots i_{n}s},C_{i_1 \dots i_{n}t}\right)> \diam \left(\bigcup
_{j_{n+1}=1}^{a_{n+1}} C_{j_1 \dots j_{n}j_{n+1}}\right).$$
\end{enumerate}
\end{definition}

\begin{remark} The only reason why the domain of $\Phi$ is $2\mathbb{N}+1$ instead of $\mathbb{N}^+$ is that
we refer to this construction in \cite{BM}, where this slight difference is important.
\end{remark}

\begin{remark} In a countable Polish space $X$ there is no balanced compact set $K\subseteq X$, since
every balanced compact set has cardinality $2^{\aleph_0}$.
\end{remark}

\begin{theorem} \label{t:ex} If $X$ is an uncountable Polish space, then there exists a balanced compact set $K\subseteq X$.
\end{theorem}

\begin{proof} Every uncountable Polish space contains a non-empty perfect subset, see \cite[(6.4) Thm.]{Ke},
so we may assume by shrinking that $X$ is also perfect.
Let us fix positive integers $a_n$ according to \eqref{00} and an index function $\Phi$ according to
the sequence $\langle a_n \rangle$.
We need to construct non-empty closed sets $C_{i_1\dots i_n}$ and positive reals $b_n$ that satisfy properties \eqref{01}-\eqref{04},
then the set $K=\bigcap _{n=1}^{\infty}\left(\bigcup_{i_1=1}^{a_1}   \cdots  \bigcup_{i_n=1}^{a_n}C_{i_1 \dots i_n} \right)$
will be a balanced compact set. Let $n\in \mathbb{N}$ and assume that $b_k$ and $C_{i_1\dots i_k}$ with $\inter C_{i_1\dots i_k}\neq \emptyset$ are already defined for all
$k\leq n$ and $(i_1,\dots,i_k)\in \mathcal{I}_k$, where we use the convention $\mathcal{I}_0=\{\emptyset\}$, $C_{\emptyset}=X$, and $b_0=\infty$. It is enough to construct $b_{n+1}$ and $C_{i_1\dots i_{n+1}}$ such that
$\inter C_{i_1\dots i_{n+1}}\neq \emptyset$ for all $(i_1,\dots, i_{n+1})\in \mathcal{I}_{n+1}$.

We define distinct points $x_{i_1\dots i_{n+1}}\in \inter C_{i_1\dots i_{n}}$ for all
$(i_1,\dots ,i_{n+1})\in \mathcal{I}_{n+1}$. First assume that $n$ is even. As $X$ is perfect and $\inter C_{i_1\dots i_{n}}\neq \emptyset$, we can fix distinct points
$x_{i_1\dots i_{n+1}}\in \inter C_{i_1\dots i_{n}}$ for all $(i_1,\dots ,i_{n+1})\in \mathcal{I}_{n+1}$.
Now assume that $n$ is odd. First consider those $(i_1,\dots, i_n)$ for which $C_{i_1\dots i_{n}}\subseteq C_{\Phi(n)}$,
then let us fix distinct points $x_{i_1\dots i_{n+1}}\in \inter C_{i_1\dots i_{n}}$ for all
$i_{n+1}\in \{1,\dots,a_{n+1}\}$. Let $\delta$ be the minimum distance between the points $x_{i_1 \ldots i_{n+1}}$ we have defined so far.
Now consider those $(i_1, \ldots, i_{n})$ for which $C_{i_1\dots i_{n}}\nsubseteq C_{\Phi(n)}$. For each of them, fix distinct points
$x_{i_1\dots i_{n+1}}\in \inter C_{i_1\dots i_{n}}$ for all $i_{n+1}\in \{1,\dots,a_{n+1}\}$ such that
$$ \diam \left(\bigcup_{i_{n+1}=1}^{a_{n+1}} \{x_{i_1\dots i_{n+1}}\}\right)\leq \frac{\delta}{2}.$$
For $(i_1,\dots,i_{n+1})\in \mathcal{I}_{n+1}$ consider the non-empty closed sets
$$C_{i_1\dots i_{n+1}}=B\left(x_{i_1\dots i_{n+1}}, b_{n+1}/2\right),$$
where $b_{n+1}>0$ is sufficiently small. Then the sets $C_{i_1\dots i_{n+1}}$ satisfy properties \eqref{01}-\eqref{04},
and clearly $\inter C_{i_1\dots i_{n+1}}\neq \emptyset$ for all $(i_1,\dots, i_{n+1})\in \mathcal{I}_{n+1}$.
\end{proof}

\begin{fact} \label{f:zero}
If $K\subseteq \mathbb{R}$ is a balanced compact set, then $K$ has zero Lebesgue measure.
\end{fact}

\begin{proof} For all $n\in \mathbb{N}^+$ and $(i_1,\dots,i_n)\in \mathcal{I}_n$ let $I_{i_1\dots i_n}\subseteq \mathbb{R}$ be compact intervals such that
$C_{i_1\dots i_n}\subseteq I_{i_1\dots i_n}$ and $\diam I_{i_1\dots i_n}=\diam C_{i_1\dots i_n}$.
Set $I^{*}_n=\bigcup_{i_1=1}^{a_1}   \cdots \bigcup_{i_n=1}^{a_n}I_{i_1\dots i_n}$.
Properties \eqref{02} and \eqref{03} imply that $\lambda(I^{*}_{n+1})\leq \lambda(I^{*}_n)/2$ for all $n\in \mathbb{N}^+$,
thus $K\subseteq \bigcap_{n=1}^{\infty} I^{*}_n$ has zero Lebesgue measure.
\end{proof}

\section{Balanced compact sets admit exact continuous gauge functions} \label{s:gauge}

The main goal of this section is to prove Theorem \ref{t:gauge}.

\bigskip

Assume that $X$ is a Polish space and $K\subseteq X$ is a
fixed balanced compact set. Let $a_n$, $b_n$, $C_{i_1\dots i_n}$, $\Phi$ be the objects witnessing
that $K$ is balanced according to Definition \ref{d:balanced}.

\begin{definition} \label{d:elem} Let $K_{i_1 \dots i_n}=K\cap C_{i_1 \dots i_n}$ for all $(i_1, \dots, i_n)\in \mathcal{I}_{n}$
and $n\in \mathbb{N}^+$. These sets are called the \emph{$n$th level elementary
pieces} of $K$. For a set $A\subseteq K$ we call the $n$th level elementary pieces of $A$ the $n$th level elementary pieces of $K$ that intersect $A$.
\end{definition}

\begin{theorem} \label{t:gauge} There exists a continuous gauge function
$h$ with $\mathcal{H}^{h}(K)=1$. Moreover,
$$\mathcal{H}^{h}(K_{i_1 \dots i_n})=\frac{1}{a_1 \cdots a_n}$$
for all $n\in \mathbb{N}^{+}$ and $(i_{1},\dots,i_{n})\in \mathcal{I}_{n}$.
\end{theorem}

\begin{proof} Consider $h\colon [0,\infty)\to [0,\infty)$,
\begin{equation} \label{d:h} h(x)= \begin{cases} 1 & \textrm{ if } x\geq 2b_1, \\
\frac{1}{a_{1}\cdots a_{n}} & \textrm{ if } 2b_{n+1}\leq x \leq
b_{n}
\textrm{ for all } n\in \mathbb{N}^{+},\\
\textrm{linear } & \textrm{ if } b_{n}\leq x \leq 2b_{n} \textrm{
for all } n\in \mathbb{N}^{+}, \\
0 & \textrm{ if } x=0.
\end{cases}
\end{equation}
As $a_n\geq 2$ for all $n\in \mathbb{N}^{+}$, properties \eqref{01}-\eqref{03} yield that $2b_{n+1}<b_{n}$ for all $n\in \mathbb{N}^+$.
Thus $b_{n}<b_{1}/2^{n-1}\to 0$ as $n\to \infty$. These imply that $h$ is well-defined. Clearly, $h$ is
non-decreasing, continuous, and $h(x)=0$ iff $x=0$. Therefore $h$ is a
continuous gauge function. It is enough to prove that $\mathcal{H}^{h}(K)=1$,
because applying the same argument for $K_{i_1 \dots i_n}$ yields
the more general statement. Then $K\subseteq \bigcup
_{i_1=1}^{a_1}   \cdots  \bigcup _{i_n=1}^{a_n} C_{i_1 \dots i_n}$ and
$\diam C_{i_1 \dots i_n}\leq b_n$ imply
$$\mathcal{H}^{h}_{b_{n}}(K) \leq \sum_{i_1=1}^{a_1}  \cdots
\sum_{i_n=1}^{a_n} h\left(\diam C_{i_1 \dots i_n}\right)\leq
a_{1}\cdots a_{n}h(b_n)=1.$$
Since $b_{n}\to 0$ as $n\to \infty$, we obtain $\mathcal{H}^{h}(K)=
\lim_{n\to \infty} \mathcal{H}^{h}_{b_n}(K)\leq 1$.

For the opposite direction assume that $K\subseteq
\bigcup_{j=1}^{\infty} U_j$, it is enough to prove that
$\sum_{j=1}^{\infty} h\left(\diam U_j\right)\geq 1$. By the
continuity of $h$ we may assume that the $U_j$'s are non-empty open,
and the compactness of $K$ implies that there is a finite subcover $K\subseteq \bigcup_{j=1}^{k} U_j$. Let us fix $m \in \mathbb{N}$ such that $2b_{m}<\min_{1\leq j\leq k}  \diam U_j$.
For all $j\in \{1,\dots,k\}$ consider
$$s_j=\# \left\{(i_1,\dots,i_{m})\in \mathcal{I}_{m}: U_j\cap K_{i_1\dots
i_{m}}\neq \emptyset\right\}.$$
Since $K\subseteq \bigcup_{j=1}^{k}U_j$, we have
\begin{equation} \label{eq:sumsj} \sum_{j=1}^{k} s_{j}\geq
a_{1}\cdots a_{m}.
\end{equation}
Now we show that for all $j\in \{1,\dots,k\}$
\begin{equation} \label{eq:Uj} h\left(\diam U_j\right)\geq \frac{s_j}{a_1\cdots a_{m}}.
\end{equation}
Let us fix $j\in \{1,\dots, k\}$. If $\diam U_j\geq 2b_1$ then $
h\left(\diam U_j\right)=1$ and $s_j\leq a_1\cdots a_{m}$ imply
\eqref{eq:Uj}. Thus we may assume that there is an $1\leq n<m$ such
that $2b_{n+1}\leq \diam U_j \leq 2b_n$. On the one hand, \eqref{03} implies that $U_j$ can intersect at most one
$n$th level elementary piece of $K$, that is, $s_j\leq a_{n+1}\cdots
a_{m}$. On the other hand, the definition of $h$ implies $h\left(\diam
U_j\right)\geq \frac{1}{a_1 \cdots a_n}$. Therefore \eqref{eq:Uj} holds.
Finally, \eqref{eq:Uj} and \eqref{eq:sumsj} yield
$$\sum_{j=1}^{k} h\left(\diam U_j\right)\geq \sum_{j=1}^{k} \frac{s_j}{a_1\cdots
a_{m}}\geq 1,$$
and the proof is complete.
\end{proof}

\begin{remark} Note that property~(\ref{04}) and the notion of an index function $\Phi$ are not needed for the proof of Theorem~\ref{t:gauge}.
We used only the natural condition $a_n\geq 2$ $(n\in \mathbb{N}^{+})$ instead of property \eqref{00}.
\end{remark}

\begin{fact} \label{f:abs} Let $K\subseteq \mathbb{R}$ be a balanced compact set,
and let $h$ be the gauge function for $K$ according to \eqref{d:h}. Then $\lambda$ is absolutely continuous for $\mathcal{H}^{h}$.
\end{fact}

\begin{proof}
Let $I$ be a compact interval such that $\bigcup_{i_1=1}^{a_1} C_{i_1}\subseteq I$, and assume $\diam I=c$. Set $g(x)=x/c$. First we prove that
$h(x)\geq g(x)$ for all $x\in [0,b_1]$. Let $n\in \mathbb{N}^{+}$. On the one side, the definition of $h$ implies $h(b_n)=\frac{1}{a_1\cdots a_n}$.
On the other side, \eqref{03} yields
$2b_{n}\left(\# \mathcal{I}_{n}-1\right)\leq \diam I$, so $b_n\leq \frac{\diam I}{2(\# \mathcal{I}_{n}-1)}\leq \frac{c}{a_1\cdots a_n}$.
Thus $h(b_n)\geq b_n/c=g(b_n)$. As $h$ is concave and $g$ is linear on
$[b_{n+1},b_n]$ for all $n\in \mathbb{N}^+$, we have $h(x)\geq g(x)$ for all $x\in [0,b_1]$.

Finally, $h|_{[0,b_1]}\geq g|_{[0,b_1]}$ implies that for all $A\subseteq \mathbb{R}$ we have
$\mathcal{H}^{h}(A)\geq \mathcal{H}^{g}(A)=\lambda(A)/c$, so $\lambda$ is absolutely continuous for $\mathcal{H}^h$.
\end{proof}

\section{The proof of Theorem \ref{t:cont}} \label{s:proof}

The goal of this section is to prove the following theorem.

\begin{theorem} \label{t:cont} Let $X$ be a Polish space, and let
$K\subseteq X$ be a balanced compact set. Then there exists a continuous gauge function $h$ such that
$0<\mathcal{H}^{h}(K)<\infty$, and for every weak contraction $f\colon K\to X$ we have $\mathcal{H}^{h}
\left(K\cap f(K)\right)=0$.
\end{theorem}

\begin{proof} Let $a_n$, $b_n$, $C_{i_1\dots i_n}$, $\Phi$ be the objects witnessing
that $K$ is balanced according to Definition \ref{d:balanced}. Let $h$ be the continuous gauge
function for $K$ according to \eqref{d:h}. Theorem \ref{t:gauge} implies $\mathcal{H}^h(K)=1$. Let $f\colon K\to X$ be a weak
contraction, it is enough to prove that $\mathcal{H}^{h} \left(K\cap f(K)\right)=0$. For all $n\in \mathbb{N}^{+}$ let
$$A_{n}=\bigcup_{i_1=1}^{a_1}   \cdots  \bigcup_{i_n=1}^{a_n} \left( K_{i_1\dots i_n}
\cap f\left (K\setminus K_{i_1 \dots i_n}\right) \right).$$
First we prove
\begin{equation} \label{fix}
K\cap f(K) \subseteq \Fix(f) \cup \bigcup_{n=1}^{\infty} A_n,
\end{equation}
where  $\Fix (f)=\{x\in K: f(x)=x\}$. Assume that  $y \in K\cap f(K)$ and
$y\notin \Fix(f)$, we need to prove that $y\in \bigcup_{n=1}^{\infty} A_n$. There is an $x\in K$ such that $f(x)=y$ and $x\neq y$.
Then $\diam K_{i_1\dots i_n} \leq b_n$ and $b_n \to 0$ imply that there is an
$n\in \mathbb{N}^{+}$ and $(i_1,\dots,i_n)\in \mathcal{I}_{n}$ such that $y\in K_{i_1\dots i_n}$ and $x\in K\setminus K_{i_1 \dots i_n}$, so $y\in A_n$. Thus $y\in \bigcup_{n=1}^{\infty} A_n$, hence
\eqref{fix} holds.

As $f$ is a weak contraction, $\Fix(f)$ has at most $1$ element. Therefore \eqref{fix} implies that it is enough to prove that $\mathcal{H}^{h} \left(\bigcup_{n=1}^{\infty} A_n \right)=0$.
Property \eqref{01} easily yields that $A_{n}\subseteq A_{n+1}$ for all $n\in \mathbb{N}^{+}$, so it is enough to prove that
\begin{equation} \label{A_n}
\lim_{n\to \infty} \mathcal{H}^{h}(A_n)=0.
\end{equation}

Let us fix $n\in \mathbb{N}^{+}$ and $(i_1,\dots, i_n)\in \mathcal{I}_{n}$. The definition of $\Phi$ yields that there is an odd number $m \geq n$ such that
$\Phi(m)=(i_1,\dots, i_n)$. Let us denote by $\Delta_{m}$ the set of $m$th level elementary pieces of $K\setminus K_{i_1\dots i_n}$. Set $E\in \Delta_m$. As $f$ is a weak contraction, $\diam f(E)\leq \diam E$. Therefore \eqref{04} together with \eqref{02} and \eqref{03} imply that $f(E)$ can intersect at most one $m+1$st level elementary piece of
$K_{i_1\dots i_n}$. Thus $f\left(\bigcup \Delta_m \right)=f(K\setminus K_{i_1 \dots i_n})$ can intersect at most $\# \Delta_m\leq a_1\cdots a_m$ many $m+1$st level elementary pieces of
$K_{i_1\dots i_n}$. Theorem \ref{t:gauge} yields that every $m+1$st level elementary piece of $K$ has $\mathcal{H}^{h}$ measure $1/(a_1\cdots a_{m+1})$, and $m\geq n$ implies $a_{m+1}\geq a_{n+1}$. Therefore
\begin{equation} \label{Ki}
\mathcal{H}^{h}  \left( K_{i_1\dots i_n} \cap f\left (K\setminus K_{i_1 \dots i_n}\right)\right)\leq \frac{a_1\cdots a_{m}}{a_1\cdots a_{m+1}} = \frac{1}{a_{m+1}}\leq \frac{1}{a_{n+1}}.
\end{equation}
Finally, equation \eqref{Ki}, the definition of $A_n$,
the subadditivity of $\mathcal{H}^{h}$, and property \eqref{00} yield
$$\mathcal{H}^{h}(A_n)\leq \frac{a_1\cdots a_{n}}{a_{n+1}}\leq \frac{1}{n}.$$
Thus \eqref{A_n} follows, and
the proof is complete.
\end{proof}

\section{The proof of our Main Theorems} \label{s:main}

Let us recall that the main goal of our paper is to answer the following question.

\begin{qq} What are the lowest possible Borel classes of $\mathbb{R}^n$ having
a non-closed element $X$ such that every weak contraction
$f\colon X\to X$ is constant?
\end{qq}

If $n\geq 2$ then the answer is $\Delta^0_2$, and there is no non-closed $\Delta^0_2$ example in $\mathbb{R}$,
see the Introduction. If $n=1$ then the following theorems show that $\Sigma^0_2$ and $\Pi^0_2$ are the lowest possible
Borel classes satisfying Question \ref{q:Elekes}.

\begin{theorem}[Main Theorem, $F_{\sigma}$ case] \label{t:main1} There exists a non-closed $F_{\sigma}$
set $F\subseteq \mathbb{R}$ such that every weak contraction $f\colon F\to
F$ is constant.
\end{theorem}

\begin{proof}  By Theorem \ref{t:ex} there exists a balanced compact set $K\subseteq \mathbb{R}$. Let $a_n$ be the positive integers and let $h$ be the continuous gauge
function for $K$ according to Definition \ref{d:balanced} and equation \eqref{d:h}, respectively.
Set $\mathbb{Q}=\{q_{n}: n\in \mathbb{N}^+\}$. Fix $z_0 \in K$ arbitrarily and for all $n\in \mathbb{N}^+$ let $K^{*}_{n}$ be the $n$th level elementary piece of $K$ containing $z_0$ (see Definition \ref{d:elem}). Consider
\begin{equation} \label{eq:F0}
F_{0}=\bigcup_{n=1}^{\infty} \left(K^{*}_{n}+q_n\right).
\end{equation}
Clearly, $F_{0}$ is an $F_{\sigma}$ set, thus $\mathcal{H}^h$ measurable.
The countable subadditivity and the translation invariance of
$\mathcal{H}^h$, and Theorem \ref{t:gauge} imply
\begin{align*}\mathcal{H}^h(F_0)&\leq \sum_{n=1}^{\infty} \mathcal{H}^h\left(K^{*}_{n}+q_n\right)= \sum_{n=1}^{\infty} \mathcal{H}^h\left(K^{*}_{n}\right)\\
&= \sum_{n=1}^{\infty} \frac{1}{a_1\cdots a_n} \leq
\sum_{n=1}^{\infty} \frac{1}{2^n}=1.
\end{align*}
As $F_0$ is a $\mathcal{H}^h$-measurable set with finite measure, there is a $G_{\delta}$ set
$G_0\subseteq \mathbb{R}$ such that
\begin{equation} \label{eq:G0} F_0\subseteq G_0 \quad \textrm{and} \quad \mathcal{H}^{h}(G_0\setminus F_0)=0,
\end{equation}
see \cite[Thm. 27.]{Ro} for the proof.
Set $F=\mathbb{R} \setminus G_0$. Clearly, $F$ is an $F_{\sigma}$ set. First we prove that $F$ is non-closed.
Fact \ref{f:zero} yields $\lambda(K)=0$, so the translation
invariance and the countable subadditivity of the Lebesgue measure
imply $\lambda(F_0)=0$. Fact \ref{f:abs} and \eqref{eq:G0} imply
$\lambda\left(G_0\setminus F_0\right)=0$. Hence $\lambda(G_0)=0$. Therefore $G_0\neq \emptyset$ yields that $G_0$ is not open,
so $F=\mathbb{R}\setminus G_0$ is non-closed. These imply also that $F$ is of full Lebesgue measure,
therefore it is dense in $\mathbb{R}$.

Assume to the contrary that there exists a non-constant weak
contraction $f\colon F\to F$. As $F$ is dense in $\mathbb{R}$, $f$ has a unique
$1$-Lipschitz extension $\widehat{f}\colon \mathbb{R}\to \mathbb{R}$. First we
prove that $\widehat{f}$ is a weak contraction. Assume to the
contrary that there are $a,b\in \mathbb{R}$, $a<b$ such that
$\big|\widehat{f}(b)-\widehat{f}(a)\big|=|b-a|$. Since $\widehat{f}$
is $1$-Lipschitz, for all $x,y\in [a,b]$ we have
\begin{equation} \label{eq:x}
\big|\widehat{f}(x)-\widehat{f}(y)\big|=|x-y|.
\end{equation}
Since $F$ is dense in $\mathbb{R}$, there are $x_0,y_0\in
F\cap[a,b]$, $x_0\neq y_0$. Applying \eqref{eq:x} for $x_0,y_0$
contradicts that $f$ is a weak contraction. Thus $\widehat{f}$ is a
weak contraction.

As $f$ is non-constant, $I=\widehat{f}(\mathbb{R})$ is a non-degenerate
interval. Then $\widehat{f}(F)=f(F)\subseteq F$ and the definition of $F$ implies $F_0\cap I\subseteq  I\setminus
F\subseteq \widehat{f}(\mathbb{R}\setminus F)
=\widehat{f}(G_0)$, so
\begin{equation} \label{eq:K+Q1} F_0\cap I\subseteq F_0\cap \widehat{f}(G_0).
\end{equation}
Equation \eqref{02} and $b_n\to 0$ yield $\diam K^{*}_{n}\to 0$ as $n\to \infty$. Thus $z_0\in K^{*}_{n}$ implies that there exists an $n\in \mathbb{N}^+$ such that $K^{*}_{n}+q_n\subseteq I$, and Theorem \ref{t:gauge} implies $\mathcal{H}^{h}\left(K^{*}_{n}\right)>0$. Therefore the translation invariance of $\mathcal{H}^{h}$ yields
\begin{equation} \label{eq:pos}
\mathcal{H}^{h}(F_0\cap I)\geq \mathcal{H}^{h}\left(K^{*}_{n}+q_n \right)=\mathcal{H}^{h}\left(K^{*}_{n}\right)>0.
\end{equation}
Theorem \ref{t:cont} implies that for all $p,q\in \mathbb{Q}$ we have
$\mathcal{H}^{h}\big((K+p)\cap \widehat{f}\left(K+q\right)\big)=0$,
as $\widehat{f}\left(K+q\right)$ is a weak contractive image of $K+p$. Therefore
$F_0\subseteq K+\mathbb{Q}$ and the countable subadditivity of $\mathcal{H}^{h}$
yield
\begin{align} \label{eq:K+Q2}
\mathcal{H}^{h}\left(F_0\cap \widehat{f}(F_0)\right)&\leq
\mathcal{H}^{h}\left((K+\mathbb{Q})\cap \widehat{f}(K+\mathbb{Q})\right) \notag \\
&\leq \sum_{p,q\in \mathbb{Q}} \mathcal{H}^{h} \left((K+p)\cap
\widehat{f}\left(K+q\right)\right) \\
&=0. \notag
\end{align}
As $\widehat{f}$ is a weak contraction and \eqref{eq:G0} holds, we
obtain
\begin{equation} \label{eq:K+Q3}
\mathcal{H}^{h}\left(\widehat{f}\left(G_0\setminus F_0\right)\right)\leq
\mathcal{H}^{h}\left(G_0\setminus F_0\right)=0.
\end{equation}
Finally, equations \eqref{eq:pos}, \eqref{eq:K+Q1}, the subadditivity of $\mathcal{H}^{h}$, \eqref{eq:K+Q2}, and \eqref{eq:K+Q3}
imply
\begin{align*} 0&<\mathcal{H}^{h}\left(F_0\cap I\right)\leq \mathcal{H}^{h}\left(F_0\cap
\widehat{f}(G_0)\right) \\
&\leq \mathcal{H}^{h}\left(F_0\cap
\widehat{f}(F_0)\right)+\mathcal{H}^{h}\left(\widehat{f}\left(G_0\setminus
F_0\right)\right)=0.
\end{align*}
This is a contradiction, so the proof is complete.
\end{proof}

\begin{theorem}[Main Theorem, $G_{\delta}$ case] \label{t:main2} There exists a non-closed $G_{\delta}$
set $G\subseteq \mathbb{R}$ such that every weak contraction $f\colon G\to
G$ is constant.
\end{theorem}

\begin{proof} Let $G=\mathbb{R}\setminus F_0$, for the definition of $F_0$ see \eqref{eq:F0}. Clearly, $G$ is a
$G_{\delta}$ set. Since $\lambda (F_0)=0$, we obtain that $G$ is of full Lebesgue measure,
thus it is non-closed and dense in $\mathbb{R}$.
Assume to the contrary that $f\colon
G\to G$ is a non-constant weak contraction. Now the argument can be completed by replacing $F$ and $G_0$ in the proof of
Theorem \ref{t:main1} by $G$ and $F_0$, respectively. Notice that $F_0$ remains $F_0$, e. g. $G_0\setminus F_0$
becomes $F_0\setminus F_0=\emptyset$. The reason of this asymmetry is that we do not consider $G_\delta$ hulls as in \eqref{eq:G0},
which makes things a little bit easier.
\end{proof}

\subsection*{Acknowledgements}
The author is indebted to A. M\'ath\'e for some helpful suggestions. He also pointed out that the original
proof of Theorem \ref{t:cont} can be significantly shortened.

\end{document}